\documentclass[reqno,a4paper, 11pt]{amsart}

\usepackage[pdfpagelabels,hidelinks,colorlinks]{hyperref}
\usepackage{graphicx}
\usepackage{latexsym}
\usepackage{mathabx}
\usepackage{amsmath}
\usepackage{amssymb}
\usepackage[T1]{fontenc}
\usepackage[utf8]{inputenc}
\usepackage{lmodern}
\usepackage{verbatim}
\usepackage{microtype}

\let\textringaccent\r

         \def\r{\rho}         \def\z{\zeta}
         
\def\a{\alpha}

\newcommand{\hol}{{\mathcal H}}
\DeclareMathOperator{\og}{O}

\def\D{{\mathbb D}}

\def\C{{\mathbb C}}  \def\N{{\mathbb N}}
 
\def\H{{\mathcal H}}

\def\tp{{\tilde{p}}}

\def\({\left(}       \def\){\right)}

\newcommand{\h}{\mathcal{H}}
\newcommand{\hti}{\widetilde{\mathcal{H}}}
\newcommand{\hg}{\mathcal{H}_{g}}

\newtheorem{theorem}{Theorem}
\newtheorem{lemma}{Lemma}
\newtheorem{corollary}{Corollary}

\theoremstyle{definition}

\theoremstyle{remark}
\newtheorem{remark}{Remark}
\numberwithin{equation}{section}
\theoremstyle{plain}
\newtheorem{other}{\bf Theorem}              

\newenvironment{Prf}{\noindent{\emph{Proof of}}}{$\Box$ }
 \DeclareMathOperator{\ogr}{O}

\begin{document}
\title[Generalized Hilbert Operators]
{ The Boundedness Problem for Generalized Hilbert Operators on Hardy Spaces}
\author[D. Norrbo]{David Norrbo}
\address{Department of Natural and Health Sciences, \textringaccent{A}bo Akademi University, FI-20500 \textringaccent{A}bo, Finland}
\email{david.norrbo@abo.fi}

\author[J. \'A. Pel\'aez]{Jos\'e \'Angel Pel\'aez}
\address{Departamento de An\'alisis Matem\'atico, Universidad de M\'alaga, Campus de Teatinos, 29071 M\'alaga, Spain}
\email{japelaez@uma.es}
\thanks{The research of  the first author is supported by the Magnus Ehrnrooth Foundation. The research of the second author is supported in part by Ministerio de Ciencia e Innovaci\'on, Spain, project PID2022-136619NB-I00; La Junta de Andaluc{\'i}a, project FQM210.}

\author[F. Wu]{Fanglei Wu}
\address{Department of Mathematics, Shantou University, Shantou, Guangdong 515063, China}
\address{Department of Physics and Mathematics, University of Eastern Finland, P.O. Box 111, FI-80101 Joensuu, Finland}
\email{fangleiwu1992@gmail.com}
\email{fanglei.wu@uef.fi}


\begin{abstract}
Let $g$ be analytic in the unit disc and consider the generalized Hilbert operator
$$
\hg(f)(z)=\int_0^1 f(t)g'(tz)\, dt.
$$
The boundedness of $\mathcal H_g$ on $H^p$ is characterized by the mean Lipschitz condition $g\in\Lambda\left(p,\frac{1}{p}\right)$ when $1<p\leq2$, while the problem remains open for $2<p<\infty$. It has been recently proved that the condition $g\in\Lambda\left(p,\frac{1}{p}\right)$ does not imply the boundedness of $\hg$ on $H^p$, $2<p<\infty$  \cite{GuoTang2026}.
We show that this condition is far from sufficient in the latter range: for every $2<p<\infty$, there exists a  function $g\in\Lambda\left(p,\frac{1}{p}\right)$ such that $\mathcal H_g$ is not bounded even from $H^p$ into $H^1$.
 The main ingredient is an exact characterization of the boundedness of $\mathcal H_g:H^p\to H^2$ for all $1\leq p\leq\infty$. In particular, when $2<p<\infty$, this mapping is bounded if and only if $g'$ belongs to a certain mixed-norm space. For lacunary symbols, the same mixed-norm condition also characterizes the boundedness of $\mathcal H_g$ on $H^p$, and hence gives a complete solution of the open problem within this class of symbols. We also show that, for $1\leq q\leq\infty$, boundedness of $\mathcal H_g:H^1\to H^q$ is characterized by the condition $g'\in H^q$. We also characterize compactness of $\hg$ in the aforementioned cases.

\end{abstract}

\maketitle

\section{Introduction and main results}\label{intro}

Let $\D$ denote the unit disc in the complex plane $\C$, and let
$\H(\D)$ be the space of analytic functions on $\D$. The classical
Hilbert matrix
$$
\left(\frac{1}{n+k+1}\right)_{n,k\geq 0}
$$
induces the Hilbert operator
$$
\mathcal H(f)(z)
=\sum_{n=0}^{\infty}
\left(\sum_{k=0}^{\infty}
\frac{\widehat f(k)}{n+k+1}\right)z^n,
\qquad z\in\D.
$$
For $f\in H^1$, this operator admits the integral representation
$$
\mathcal H(f)(z)=\int_0^1\frac{f(t)}{1-tz}\,dt,
\qquad z\in\D.
$$
Motivated by this representation, Galanopoulos, Girela, Pel\'aez and
Siskakis \cite{GaGiPeSis} undertook a systematic study of the family of
generalized Hilbert operators on Hardy and standard Bergman spaces
\begin{equation}\label{H-g}
\hg(f)(z)=\int_0^1 f(t)g'(tz)\,dt,
\qquad z\in\D,
\end{equation}
where $g\in\H(\D)$. By the Fej\'er--Riesz inequality
\cite[p.~46]{D}, the integral in
\eqref{H-g} converges absolutely, and
therefore defines a function in $\H(\D)$. For
$g(z)=\log\frac{1}{1-z}$, the operator $\hg$ reduces to the classical
Hilbert operator $\mathcal H$. The generalized
Hilbert operators between classical spaces of analytic functions have
since been extensively studied
\cite{Blasco2022,GaGiPeSis,GuoTang2026,PelRathg,PelSeco}.

In order to state our main results we need to introduce some notation. For $0<p<\infty$ and $0<r<1$, set
$$
M_p(r,f)
=
\left(
\frac{1}{2\pi}\int_0^{2\pi}
|f(re^{i\theta})|^p\,d\theta
\right)^{1/p},
$$
and let
$$
M_\infty(r,f)=\max_{|z|=r}|f(z)|.
$$
For $0<p\leq\infty$, the Hardy space $H^p$ consists of those
$f\in\H(\D)$ for which
$$
\|f\|_{H^p}
=
\sup_{0<r<1}M_p(r,f)
<\infty.
$$

The extension $v(z)=v(|z|)$ of a non-negative function $v\in L^1([0,1))$, is called a radial weight. For $0<p\leq \infty$ and $0<q<\infty$, the
analytic
 weighted mixed-norm 
 spaces $A^{p,q}_v$  consists of measurable functions
$f\in \H(\D)$
 such that
	$$
	\|f\|_{A^{p,q}_v}^q =\int_0^1 M_p^q(r,f) v(r)\,dr<\infty
	$$
If $v(z)=(\alpha+1)(1-|z|^2)^{\alpha}$, $\a>-1$ then we write   $A^{p,q}_v=A^{p,q}_\a$.

We also consider  the mean Lipschitz spaces
$\Lambda\left(p,\alpha\right)$. For $1\le p<\infty$ and $0<\a\le 1$ the
space $\Lambda\left(p,\a\right)$  consists of those $g\in \hol(\D)$ having
a non-tangential limit $g(e^{i\theta})$ almost everywhere and such
that
$$
\omega_p(g, t)=\ogr(t^{\alpha}), \quad t\to 0,
$$
where
$$
\omega_p(g, t)=\sup_{0<h\le t}\left(\int_0^{2\pi}
|g(e^{i(\theta+h)})-g(e^{i\theta})|^p \frac{d\theta}{2\pi}\right)^{1/p}
$$
is the integral modulus of continuity of order $p$.
A classical result of Hardy and Littlewood
\cite{HL-32} (see also Chapter~5 of \cite{D}) asserts that
\begin{equation}\label{hl32}
\Lambda\left(p,\a\right)=\left \{ f\in H^p : M_p(r,f^\prime)
=\og \left ((1-r)^{\alpha -1}\right )\right \},
\end{equation}
for $1\leq p<\infty
,\ 0<\a \leq 1$. The corresponding "little oh" spaces are denoted by $\lambda(p,\alpha)$.

For  $1<p<\infty$ and $0<\a \leq 1$ we write $$\|g\|_{ \Lambda(p,\a)}=|g(0)|+\sup_{t>0}\frac{\omega_p(g, t)}{t^\alpha}.$$

The boundedness of $\hg$ on $H^p$ is well understood when
$1<p\leq 2$, whereas the case $2<p<\infty$ remains open. More
precisely, the condition
$g\in\Lambda\left(p,\frac{1}{p}\right)$
is necessary for the boundedness of $\hg$ on $H^p$ whenever
$1<p<\infty$, and it is also sufficient when $1<p\leq 2$
\cite[Theorem 1]{GaGiPeSis}. For $2<p<\infty$, however, this condition is no
longer sufficient
\cite[Theorem~3.5 and Corollary~3.6]{GuoTang2026}, and no complete
characterization is currently known.

Motivated by this open problem, we study several related 
properties of $\hg$ that shed some light on the case $p>2$. Our
starting point is the boundedness and compactnes of $\hg:H^p\to H^2$, which we
characterize for all $1\leq p\leq\infty$. In particular, in the range
$2<p<\infty$, the resulting condition is expressed in terms of an
analytic mixed-norm space. 

\begin{theorem}\label{thm:H2target}
Let $g\in\H(\D)$. Then the following statements hold.

\begin{itemize}

\item[(i)] The following statements are equivalent:
\begin{itemize}
\item[(a)] $\H_g:H^1\to H^2$ is bounded;
\item[(b)] $\H_g:H^1\to H^2$ is compact;
\item[(c)] $g'\in H^2$.
\end{itemize}
Moreover,
$$
2\|g'\|_{H^2}
\leq
\|\H_g\|_{H^1\to H^2}
\leq
\pi\|g'\|_{H^2}.
$$

\item[(ii)] If $1<p\leq2$, then $\H_g:H^p\to H^2$
is bounded if and only if
$g\in\Lambda\left(2,\frac1p\right)$.
Moreover,
$$
\|\H_g\|_{H^p\to H^2}
\asymp
\|g-g(0)\|_{g\in\Lambda\left(2,\frac1p\right)}.
$$
Furthermore, $
\H_g:H^p\to H^2$
is compact if and only if
$
g\in\lambda\left(2,\frac1p\right).
$
\item[(iii)] If $2<p<\infty$, then the following statements are equivalent:
\begin{itemize}
\item[(a)] $\H_g:H^p\to H^2$ is bounded;
\item[(b)] $\H_g:H^p\to H^2$ is compact;
\item[(c)] $g'\in A_{\frac{\tilde p}2}^{2,\tilde p}$, where
$\frac1{\tilde p}=
\frac12-\frac1p
$.
\end{itemize}
Moreover,
$$
\|\H_g\|_{H^p\to H^2}
\asymp
\|g'\|_{A_{\frac{\tilde p}{2}}^{2,\tilde p}}.
$$

\item[(iv)] The following statements are equivalent:
\begin{itemize}
\item[(a)] $\H_g:H^\infty\to H^2$ is bounded;
\item[(b)] $\H_g:H^\infty\to H^2$ is compact;
\item[(c)] $g\in H^2$.
\end{itemize}
Moreover,
$$
\|\H_g\|_{H^\infty\to H^2}
=
\|g-g(0)\|_{H^2}.
$$

\end{itemize}
\end{theorem}
The above characterization enables us to
construct, for every $2<p<\infty$, a symbol
$g\in\Lambda\left(p,\frac{1}{p}\right)$ such that $\hg:H^p\to H^1$ is not bounded. Hence it improves a recent result in \cite[Theorem 3.5 and Corollary 3.6]{GuoTang2026}. 

\begin{corollary}\label{th:counterexample}
Let $2<p<\infty$. Then there is $g\in\Lambda\left(p,\frac{1}{p}\right)$ such that $\hg: H^p\to H^1$ is not bounded. In particular, $\hg$ is not bounded on $H^p$.
\end{corollary}

Indeed, the counterexample above can be chosen by lacunary series. Recall that a series
$$
g(z)=\sum_{j=0}^{\infty}a_jz^{n_j}
$$
is a lacunary series if $n_{j+1}\ge\lambda n_j$ for some $\lambda>1$ and every $j$. Our next result says that the open problem can be completely answered for the subclass of lacunary series symbols $g$.

\begin{theorem}\label{th:lacunaryp-p} Assume  that $2<p<\infty$ and
$g\in \H(\D)$ is a lacunary series. Then $\hg$ is bounded on $H^p$ if and only if $g'\in A^{2,\tilde{p}}_{\frac{\tilde{p}}{2}}$ where $\frac{1}{\tp}=\frac{1}{2}-\frac{1}{p}$.
Moreover,
\begin{equation}\label{eq:norm1}
\| \hg\|_{H^p\to H^p}\asymp \|g'\|_{A^{2,\tilde{p}}_{\frac{\tilde{p}}{2}}}.
\end{equation}
\end{theorem}

We next turn to the endpoint space $H^1$. For mappings from
$H^1$ into a Hardy space, both the boundedness and compactness of $\hg$ admit a
simple characterization in terms of the derivative of the symbol. For $1\leq p <\infty$, the conditions coincide, while for the codomain $H^\infty$, compactness is characterized by $g'\in \mathcal A$, where $\mathcal A$ is the disc algebra, that is, the closure of analytic polynomials in $H^\infty$, $H^\infty\cap C(\overline{\D})$.
\begin{theorem}\label{OnH1}
   Let $g\in\H(\D)$ and $1\le p<\infty$. Then the following statements are equivalent:
   \begin{itemize}
       \item[(i)] $\H_g: H^1\to H^p$ is bounded;
       \item[(ii)]  $\H_g: H^1\to H^p$ is compact;
       \item[(iii)] $g'\in H^p$.
   \end{itemize}
Moreover, $\H_g: H^1\to H^\infty$ is bounded if and only if $g'\in H^\infty$, and compact if and only if $g'\in \mathcal A$.
 
 Furthermore, for $1\le p\leq\infty$
   \[
 2\|g'\|_{H^p}\le   \|\H_g\|_{H^1\to H^p}\le \pi\|g'\|_{H^p}.
   \]
   
\end{theorem}

We finish this introduction with a couple of words on the notation already used. The letter $C$ will denote an absolute constant whose value depends on the parameters indicated in the subscript, and may change from one occurrence to another. If there exists a constant
$C>0$ such that $a\le Cb$, we will write either $a\lesssim b$ or $b\gtrsim a$. In particular, if $a\lesssim b\lesssim a$, then we write $a\asymp b$ and say that $a$ and $b$ are comparable.

\section{Boundedness from \texorpdfstring{$H^1$ to $H^p$}{H1 to Hp}}

In this section we will prove Theorem~\ref{OnH1}.

\begin{Prf}{\em{Theorem~\ref{OnH1}.}}
    We begin by characterizing boundedness. To this end, let $1\leq p\leq \infty$ and assume that $g'\in H^p$. Then,  Fubini's theorem in the case $p=1$ and Minkowki's inequality in the case $1<p\leq\infty$, together with  F\'ejer-Riesz inequality yield
    $$
    \|\H_g(f)\|_{H^p}\leq \int_0^1 |f(t)|\|g'(t\cdot)\|_{H^p}\,dt\leq \pi\|g'\|_{H^p}\|f\|_{H^1},\quad f\in H^1.
    $$
Hence $\H_g$ is bounded from $H^1$ to $H^p$ and  $\|\H_g\|_{H^1\to H^p}\le \pi \|g'\|_{H^p}$.

Next, we assume that $\hg:H^1\to H^p$ is bounded. We consider the test functions
$$
F_\rho(z)=\frac{1-\rho^2}{(1-\rho z)^2}, \quad 0<\rho<1,
$$
which satisfy $\|F_\rho\|_{H^1}=1$. Now, we are going to prove that for every $g\in\H(\D)$
\begin{equation}\label{xxxxx}
 \lim_{\rho\to 1^-} \H_g(F_\rho)(z)=2g'(z), \quad z\in\D
\end{equation}
holds, where the convergence is uniform on compact subsets of $\D$.

Take $R\in (0,1)$ and $\varepsilon\in (0,1)$.
Since $g'$ is uniformly continuous in $\overline{D(0,R)}$,  there exists $\delta>0$ such that
\begin{equation}\label{eq:h11}
|g'(tz)-g'(z)|<\varepsilon \quad\text{when  $t\in [0,1]$, $z\in \overline{D(0,R)}$ and $1-t<\delta$. }
\end{equation}

Bearing in mind that $\int_0^1 F_\rho(t)\,dt=1+\rho$, for any $\rho \in [0,1)$ such that   $1-\varepsilon<\rho$ and $z\in \overline{D(0,R)}$
\begin{align*}
 | \H_g(F_\rho)(z)-2g'(z)|&\leq\left|\int_0^1F_\rho(t)\big(g'(tz)-g'(z)\big)\,dt\right|+|\rho-1||g'(z)|\\
 &\leq \left(\int_0^{1-\delta}+\int_{1-\delta}^1\right)\left|\frac{1-\rho^2}{(1-\rho t)^2}\big(g'(tz)-g'(z)\big)\right|\,dt+|\rho-1||g'(z)|\\
 &\leq 2\sup_{\z \in\overline{D(0,R)} }|g'(\z)|\int_0^{1-\delta} \frac{1-\rho^2}{(1-\rho t)^2}\,dt+(\rho+1)\varepsilon+(1-\rho)
|g'(z)|,
\\
 &\leq \left( \frac{4}{\delta^2}+1\right)(1-\rho)\sup_{\z \in\overline{D(0,R)} }|g'(\z)|+2\varepsilon,
\end{align*}
which after letting $\rho\to 1^-$ proves \eqref{xxxxx}. Therefore, for a fixed $0<r<1$
$$
\lim_{\rho\to1^-}M_p(r, \H_g(F_\rho))=2M_p(r,g').
$$
Noting that $\H_g$ is bounded from $H^1$ to $H^p$, we have
$$
M_p(r, \H_g(F_\rho))\leq \|\H_g(F_\rho)\|_{H^p}\leq \|\H_g\|_{H^1\to H^p}, 
 \quad 0<\rho<1.
$$
Hence, we obtain
$$
M_p(r,g')\leq \frac12\|\H_g\|_{H^1\to H^p},
$$
which implies that $g'\in H^p$.

To finish the proof for the case $1\leq p<\infty$, it suffices to show that (i) implies (ii). If $\H_g$ is bounded from $H^1$ to $H^p$, we have shown that $g'\in H^p$. Therefore, there exists a sequence of polynomials $P_n$ such that $\lim_{n\to\infty}\|g'-P_n'\|_{H^p}=0$. Consequently,
$$
\lim_{n\to\infty}\|\H_g-\H_{P_n}\|_{H^1\to H^p}=\lim_{n\to\infty}\|\H_{g-P_n}\|_{H^1\to H^p}\lesssim \lim_{n\to\infty}\|g'-P'_n\|_{H^p}=0.
$$
Notice that each $\H_{P_n}$ is finite-rank, so it is compact from $H^1$ to $H^p$. Hence $\H_g$ is also compact on $H^1$ to $H^p$.

Similarly, we obtain that $\hg:H^1 \to H^{\infty}$ is compact if $g'\in \mathcal A$, because $\mathcal A$ is the $H^{\infty}$-closure of the analytic polynomials. To see that  $g'\in \mathcal A$ are the only symbols for which $\hg:H^1 \to H^{\infty}$ is compact, we note that $\H_g(F_\rho) \in \mathcal A$ for every $0<\rho<1$. This follows from the integrand being uniformly bounded and the dominated convergence theorem. Assuming $\hg$ is compact, there is a sequence $\{\rho_n\}$ converging to $1$ such that $\lim_{n\to\infty}\H_g(F_{\rho_n})  = F$ in $H^\infty$ for some $F\in H^\infty$. Since $ \mathcal A$ is norm closed, we obtain $F\in \mathcal A$, and using \eqref{xxxxx}, we conclude $g' = F/2\in \mathcal A$.
\end{Prf}

\begin{remark}
Since $H^1$ embeds continuously in $L^1(0,1)$, a generalization of $g'\in H^p$ being sufficient for $\hg$ to be bounded $H^1 \to H^p, \ 1\leq p\leq \infty$, could be found in \cite[Proposition 3.1]{Blasco2022}.
\end{remark}

\begin{remark}
By \cite[Theorem~6.7]{D}  the Hadamard product induced by $g'(z)=\sum_{k=0}^{\infty}(k+1)\widehat{g}(k+1)z^k$ is bounded from $H^1$ and $H^2$  if and only if
\begin{equation}\label{eq:gprimeH12}
\sup_{N\in \mathbb{N}}\frac{1}{N^2}\sum_{k=0}^N(k+1)^4 |\widehat{g}{(k+1)}|^2<\infty.
\end{equation}
By considering $g(z)=\sum_{k=1}^{\infty} k^{-1/2}  z^k$, we observe that \eqref{eq:gprimeH12} is strictly weaker than the condition $g'\in H^2$. That is, 
the boundedness of $\hg: H^1\to H^2$ is sufficient but not necessary so that  the Hadamard product induced by $g'$ is bounded from $H^1$ and $H^2$. 

\end{remark}

\section{Boundedness from  $H^p$ to $H^2$, $1<p\le 2$.}
\par
We note that $\hg$ has a representation in terms of the Taylor coefficients, which will be used repeatedly throughout this manuscript. Indeed, a simple computation shows that if
$g(z)=\sum_{n=0}^\infty \widehat{g}(n)z^n\in \hol (\D )$ and
$f(z)=\sum_{n=0}^\infty \widehat{f}(n)z^n\in H^1$ then
\begin{equation}
\begin{split}\label{Hgcoef}
\mathcal{H}_g(f)(z)&=\sum_{k=0}^{\infty} \left( (k+1)\widehat{g}(k+1)\int_0^1t^k f(t)\,dt\right)z^k\\
&=\sum_{k=0}^\infty \left
((k+1)\widehat{g}(k+1)\sum_{n=0}^\infty \frac{\widehat{f}(n)}{n+k+1}\right
)z^k.
\end{split}
\end{equation}

Theorem~\ref{thm:H2target}(ii) is contained in the next result.

\begin{theorem}\label{th:hardypmenor2} Assume  that $1<p\le 2$ and
$g\in \H(\D)$. Then $\hg$ is bounded from $H^p$ to
$H^2$ if and only if $g\in \Lambda \left(2, \frac{1}{p}\right).$
Moreover,
\begin{equation}\label{eq:norm2}
\| \hg\|_{H^p\to H^2}\asymp \|g-g(0)\|_{ \Lambda \left(2, \frac{1}{p}\right).}.
\end{equation}
Furthermore, $\H_g$ is compact from $H^p$ to $H^2$ if and only if $g\in \lambda(2,1/p)$.
\end{theorem}
Let us mention that the case $p=2$ in Theorem~\ref{th:hardypmenor2} was proved in \cite[Theorem~1]{GaGiPeSis}.
In order to prove Theorem~\ref{th:hardypmenor2} we need to introduce some notation and recall some results.

Throughout the paper we shall use the following notation:
If $g(z)=\sum_{k=0}^\infty \widehat{g}(k) z^k\in\hol(\D)$ and $n\ge 0$, we set
$$
\Delta_ng(z)=\sum_{k\in I(n)}\widehat{g}(k)z^k
$$
where $I(n)=\left\{k\in\N:\, 2^n\le k\le 2^{n+1}-1 \right\}$ if $n\in \N$ and $I(0)=\{0,1\}$.
\par Let us also recall several distinct characterizations of  $\Lambda(p,\alpha)$ and $\lambda(p,\alpha)$
spaces, see \cite{BSS}, \cite{D},  and \cite{MP}.
\begin{other}[{\cite[Theorem A]{GaGiPeSis}}]\label{th:mlip}
Suppose that $1<p<\infty$, $0<\alpha<1$ and $g\in \H(\D)$. The
following conditions are equivalent:
\begin{itemize}
    \item [(i)]$g\in \Lambda(p,\a)$;
\item [(ii)] $M_p(r,g')=\ogr\left(\frac{1}{(1-r)^{1-\alpha}}\right)$,\,\,\,\,\,as $r\to 1^-$;
\item [(iii)] $||\Delta_ng||_{H^p}=\ogr\left(2^{-n\alpha}\right)$,\quad\,\,\,\, as $n\to\infty$;
\item [(iv)] $||\Delta_ng'||_{H^p}=\ogr\left(2^{n(1-\alpha)}\right)$,\quad as $n\to\infty$.
\end{itemize}
Moreover,
\begin{equation*}\begin{split}
\|g\|_{ \Lambda(p,\a)} &\asymp |g(0)|+\sup_{0\le r<1} M_p(r,g')(1-r)^{1-\alpha}
\asymp \sup_{n\in \N\cup\{0\}} 2^{n\alpha}\|\Delta_n g\|_{H^p}
\\ &  \asymp  |g(0)|+ \sup_{n\in \N\cup\{0\}} 2^{-n(1-\alpha)}\|\Delta_n g'\|_{H^p}.
\end{split}\end{equation*}
\end{other}

\begin{remark}
 The corresponding results for the little-oh space $\lambda(p,\alpha)$ remain true,
and they can be proved following the proofs in the references for Theorem \ref{th:mlip}.   
\end{remark}

 Let us consider the sublinear Hilbert operator $\hti $ defined
by
 $$
 \hti(f)(z)=\int_0^1\frac{|f(t)|}{1-tz}\,dt, \quad f\in H^1.
 $$

Observe that the Taylor coefficients of $\hti(f)$ are given by the sequence $\left\{ \int_{0}^1 t^n |f(t)|\,dt     \right\}$, which is positive
and decreasing. Then by \cite[Thorem~A]{Padec}

\begin{equation}\begin{split}\label{eq:htinorm}
\| \hti(f)\|^p_{H^p} & \asymp \sum_{k=0}^\infty   (k+1)^{p-2}\left(\int_{0}^1 t^k |f(t)|\,dt  \right)^p
\\ & \asymp \left(\int_{0}^1  |f(t)|\,dt  \right)^p+ \sum_{n=0}^\infty 2^{n(p-1)}\left(\int_{0}^1 t^{2^n} |f(t)|\,dt  \right)^p.
\end{split}\end{equation}

 The boundedness of $\widetilde{H}$ on $H^p$, $1\le p<\infty$, is a key ingredient in the proof of several results of this paper  \cite[Theorem~5]{GaGiPeSis}.

\begin{Prf}{\em{Theorem~\ref{th:hardypmenor2}.}}
Assume that $g\in  \Lambda \left(2, \frac{1}{p}\right)$.
First, bearing in mind the well-known inequality
$$M_\infty(r,f)\le C_p\|f\|_{H^p}(1-r)^{-1/p}, \quad 0\le r<1, f\in \H(\D). $$
It follows that
\begin{equation}\label{eq:moments}
\int_0^1 t^k| f(t)|\, dt \le C_p \|f\|_{H^p} \int_0^1 t^k( 1-t)^{-1/p}\le C_p \frac{ \|f\|_{H^p}}{(k+1)^{1/p'}}, \quad k \in \N\cup\{0\}.
\end{equation}
Therefore, by \eqref{eq:moments}, Theorem~\ref{th:mlip},   \eqref{eq:htinorm} and \cite[Theorem~5]{GaGiPeSis}
\begin{equation*}\begin{split}
&\| \hg(f)\|^2_{H^2}= \sum_{k=0}^\infty   (k+1)^{2}|\widehat{g}(k+1)|^2\left|\int_{0}^1 t^k f(t)\,dt  \right|^2
\\ & \le \sum_{n=0}^\infty \sum_{k\in I(n)}  (k+1)^{2}|\widehat{g}(k+1)|^2 \left(\int_{0}^1 t^k |f(t)|\,dt  \right)^2
\\ & \lesssim \|g-g(0)\|^2_{ \Lambda \left(2, \frac{1}{p}\right)} \left(\int_{0}^1  |f(t)|\,dt  \right)^2+
 \sum_{n=0}^\infty
\left(\int_{0}^1 t^{2^n} |f(t)|\,dt  \right)^2 \sum_{k\in I(n)}  (k+1)^{2}|\widehat{g}(k+1)|^2
\\ & \lesssim \|f\|^{2-p}_{H^p}\|g-g(0)\|^2_{ \Lambda \left(2, \frac{1}{p}\right)} \left(\int_{0}^1  |f(t)|\,dt  \right)^p
\\ & + \|f\|^{2-p}_{H^p}   \sum_{n=0}^\infty
\left(\int_{0}^1 t^{2^n} |f(t)|\,dt  \right)^p 2^{-n \frac{2-p}{p'}} \sum_{k\in I(n)}  (k+1)^{2}|\widehat{g}(k+1)|^2
\\ & \lesssim \|f\|^{2-p}_{H^p}\|g-g(0)\|^2_{ \Lambda \left(2, \frac{1}{p}\right)}\left[ \left(\int_{0}^1  |f(t)|\,dt  \right)^p+
\sum_{n=0}^\infty
2^{n(p-1)}\left(\int_{0}^1 t^{2^n} |f(t)|\,dt  \right)^p\right]
\\ & \lesssim \|g-g(0)\|^2_{ \Lambda \left(2, \frac{1}{p}\right)}\|f\|^{2-p}_{H^p}  \| \hti(f)\|^p_{H^p} \lesssim \|g-g(0)\|^2_{ \Lambda \left(2, \frac{1}{p}\right)}\|f\|^{2}_{H^p}, \quad f \in \H(\D).
\end{split}\end{equation*}
Therefore, $\hg: H^p \to H^2$ is bounded and  $\| \hg\|_{H^p\to H^2}\lesssim \|g-g(0)\|_{ \Lambda \left(2, \frac{1}{p}\right)}.$

\par Conversely, assume that  $\hg: H^p \to H^2$ and let us prove that $g\in  \Lambda \left(2, \frac{1}{p}\right)$. For each $n\in \N\cup\{0\}$, let us consider the function
$$f_{2^n}(z)=\frac{2^{\frac{-n}{p'}}}{1-a_{2^n}z}\, \quad a_{2^n}=1-2^{-n}.$$
A calculation shows that
$\sup_{n\in \N\cup\{0\}}\| f_{2^n}\|_{H^p}\asymp 1.$
Moreover, there is $C>0$ such that
$$\int_{0}^1 t^{k}f_{2^n}(t)\,dt \ge \int_{a_{2^n}}^1 t^{k}f_{2^n}(t)\,dt\ge C 2^{\frac{-n}{p'}}, \quad \text{for all $k\in I(n)$}.$$
Therefore, for $n\in \N\cup\{0\}$
\begin{equation*}\begin{split}
\| \hg\|^2_{H^p\to H^2}& \gtrsim \sup_{N \in \N\cup\{0\}}\sum_{k=0}^{\infty}  (k+1)^{2}|\widehat{g}(k+1)|^2\left|\int_{0}^1 t^k f_{2^N}(t)\,dt  \right|^2
\\ & \ge \sum_{k\in I(n)}  (k+1)^{2}|\widehat{g}(k+1)|^2\left|\int_{0}^1 t^k f_{2^n}(t)\,dt  \right|^2
\\ & \geq C 2^{\frac{-2n}{p'}}\sum_{k\in I(n)}  (k+1)^{2}|\widehat{g}(k+1)|^2.
\end{split}\end{equation*}
That is, $$\sup_{n\in \N\cup\{0\}} 2^{\frac{-2n}{p'}}\sum_{k\in I(n)}  (k+1)^{2}|\widehat{g}(k+1)|^2\lesssim \| \hg\|^2_{H^p\to H^2}.$$ Therefore,  by
Theorem~\ref{th:mlip},  $g\in  \Lambda \left(2, \frac1p\right)$ and $ \|g-g(0)\|_{ \Lambda \left(2, 1/p\right)}\lesssim \| \hg\|_{H^p\to H^2} $.

Next, we are going to prove the compactness part. Assume first that
$g\in  \lambda \left(2, \frac1p\right)$.
Since it is the closure of the analytic polynomials in
$g\in  \Lambda \left(2, \frac1p\right)$, there exists a sequence of polynomials $\{P_n\}$ such
that
$$
\|(g-P_n)-(g(0)-P_n(0))\|_{\Lambda(2,1/p)}
\to0,\quad n\to\infty.
$$
By the boundedness estimate,
\begin{align*}
\|\H_g-\H_{P_n}\|_{H^p\to H^2}
&=
\|\H_{g-P_n}\|_{H^p\to H^2}\\
&\lesssim
\|(g-P_n)-(g(0)-P_n(0))\|_{\Lambda(2,\frac{1}{p})}
\to0,\quad n\to\infty.
\end{align*}
Since each $\H_{P_j}$ has finite rank, it follows that
$\H_g:H^p\to H^2$ is compact.

Conversely, suppose $\H_g$ is compact. Then for the same bounded sequence $\{f_{2^n}\}$ in $H^p$ considered in the boundedness part, it is easy to see that $\lim_{n\to\infty}f_{2^n}(z)=0$ on any compact subset of $\D$. Therefore, the compactness of $\H_g$ implies that $\lim_{n\to\infty}\|\H_gf_{2^n}\|_{H^2}=0$. Then the same steps as in the boundedness part show that 
$$
\lim_{n\to\infty} 2^{\frac{-2n}{p'}}\sum_{k\in I(n)}  (k+1)^{2}|\widehat{g}(k+1)|^2\lesssim \lim_{n\to\infty} \| \hg f_{2^n}\|^2_{ H^2}=0,
$$
which means $g\in\lambda(2,1/p)$.
This finishes the proof.
\end{Prf}

\section{Boundedness from  $H^p$ to $H^2$, $2<p<\infty$.}

Theorem~\ref{thm:H2target}(iii) is contained in the next result.

\begin{theorem}\label{th:hardypmayor2} 
Assume  that $2<p<\infty$ and
$g\in \H(\D)$.
Then the following statements are equivalent:
\begin{itemize}
\item[(a)] $\H_g:H^p\to H^2$ is bounded;
\item[(b)] $\H_g:H^p\to H^2$ is compact;
\item[(c)] $g'\in A_{\frac{\tilde p}2}^{2,\tilde p}$, where
$\frac1{\tilde p}=
\frac12-\frac1p
$.
\end{itemize}
Moreover,
\begin{equation}\label{eq:norm1}
\| \hg\|_{H^p\to H^2}\asymp \|g'\|_{A^{2,\tilde{p}}_{\frac{\tilde{p}}{2}}}.
\end{equation}

\end{theorem}

To prove Theorem~\ref{th:hardypmayor2}, we need the following technical result.

\begin{lemma}\label{le:technical}
Let $0<q,t,s<\infty$, $1<p<\infty$ and $0\le \rho<1$. Let us consider the functions
$$\phi_\rho(r)=M_q^s(r,g'_\rho)(1-r)^t, \quad 0\le r<1,$$
and
$$Q_\rho(z)=\int_0^1 \frac{\phi_\rho(x)}{1-xz}\,dx, \quad z\in \D.$$
Then, there exists $C>0$ (which do not depend on $\rho$) such that
$$Q_\rho(r)\ge C \phi_\rho(r),  \quad 0\le r<1,$$
and
$$ \| Q_\rho\|_{H^p}\asymp \| \phi_\rho\|_{L^p([0,1))},$$
where the constants involved do not depend on $\rho$.
\end{lemma}
\begin{proof}
First,
$$
Q_\rho(r)\ge \int_r^1 \frac{\phi_\rho(x)}{1-xr}\,dx\ge  \frac{M^s_q(r,g'_\rho)}{1-r^2} \int_r^1 (1-x)^t\, dx\ge C  \phi_\rho(r).
$$
Moreover, by the F\'ejer-Riesz inequality \cite[Theorem 3.13]{D} and the above inequality
$$ \| Q_\rho\|_{H^p}\gtrsim  \| Q_\rho\|_{L^p([0,1))} \gtrsim \| \phi_\rho\|_{L^p([0,1))}.$$
Since $1<p<\infty$, then $(H^{p'})^\star\simeq H^p$ under the $H^2$-pairing (with equivalence of norms) \cite[Theorem 9.8]{Zhu}. So,
using again the F\'ejer-Riesz inequality
\begin{equation*}\begin{split}
 \| Q_\rho\|_{H^p} & \asymp \sup_{\|f\|_{H^{p'}}=1} |\langle f, Q_\rho \rangle_{H^2}|
\\ & = \sup_{\|f\|_{H^{p'}}=1}  \left| \frac{1}{2\pi} \int_0^{2\pi} f(e^{i\theta}) \overline{Q_\rho(e^{i\theta})} \, d\theta \right|
\\ & = \sup_{\|f\|_{H^{p'}}=1}  \left| \frac{1}{2\pi} \int_0^{2\pi} f(e^{i\theta}) \overline{\int_0^1 \frac{\phi_\rho(x)}{1-xe^{i\theta}}\,dx} \, d\theta \right|
\\ & = \sup_{\|f\|_{H^{p'}}=1}  \left| \frac{1}{2\pi} \int_0^{1} \phi_\rho(x) \int_0^{2\pi} \frac{ f(e^{i\theta})}{1-xe^{-i\theta}}\,d\theta \, dx \right|
\\ & = \sup_{\|f\|_{H^{p'}}=1} \left|\int_0^{1} \phi_\rho(x) f(x)\,dx\right|\le \| \phi_\rho\|_{L^p([0,1))}  \| f\|_{L^{p'}([0,1))}\lesssim\| \phi_\rho\|_{L^p([0,1))}.
\end{split}\end{equation*}
The use of Fubini's theorem is justified by absolute integrability due to $\phi_\rho(x) \lesssim (1-x)^t$, where the constant is independent of $x$. This finishes the proof.
\end{proof}

Now we are ready to prove Theorem~\ref{th:hardypmayor2}.

\begin{Prf}{\em{Theorem~\ref{th:hardypmayor2}.}}
Assume that $g'\in A^{2,\tilde{p}}_{\frac{\tilde{p}}{2}}$. First, observe that
by \cite[Theorem~1]{MP}
\begin{equation}\begin{split}\label{eq:symbolnorm}
\|g'\|^{\tp}_{A^{2,\tilde{p}}_{\frac{\tilde{p}}{2}}} & \asymp \int_0^1 M_2^{\tp} (r,g')  (1-r)^{ \frac{ \tilde{p} }{2} }\,dr
\\ &\asymp \sum_{n=0}^\infty 2^{-n\left( \frac{\tp}{2}+1 \right)}
\left( \sum_{j\in I(n)}(j+1)^2|\widehat{g}(j+1)|^2 \right)^ {\frac{\tp}{2}}
 = \sum_{n=0}^\infty 2^{-n\left( \frac{\tp}{2}+1 \right)} \| \Delta_n g' \|^{\tp}_{H^2}.
\end{split}\end{equation}
Therefore, by \eqref{eq:symbolnorm},  H\"older's inequality with $p/2$,  \eqref{eq:htinorm} and \cite[Theorem~5]{GaGiPeSis}
\begin{equation*}\begin{split}
&\| \hg(f)\|^2_{H^2}= \sum_{k=0}^\infty   (k+1)^{2}|\widehat{g}(k+1)|^2\left|\int_{0}^1 t^k f(t)\,dt  \right|^2
\\ & \le \sum_{n=0}^\infty \sum_{k\in I(n)}  (k+1)^{2}|\widehat{g}(k+1)|^2 \left(\int_{0}^1 t^k |f(t)|\,dt  \right)^2
\\ & \lesssim \|g'\|^{2}_{A^{2,\tilde{p}}_{\frac{\tilde{p}}{2}}}\left(\int_{0}^1  |f(t)|\,dt  \right)^2+
 \sum_{n=0}^\infty
\left(\int_{0}^1 t^{2^n} |f(t)|\,dt  \right)^2 \sum_{k\in I(n)}  (k+1)^{2}|\widehat{g}(k+1)|^2
\\ & \lesssim \|g'\|^{2}_{A^{2,\tilde{p}}_{\frac{\tilde{p}}{2}}}\left(\int_{0}^1  |f(t)|\,dt  \right)^2
\\ & + \left( \sum_{n=0}^\infty
\left(\int_{0}^1 t^{2^n} |f(t)|\,dt  \right)^p 2^{n(p-1)} \right)^{\frac{2}{p}}
 \left(  \sum_{n=0}^\infty 2^{-n\left( \frac{\tp}{2}+1 \right)}
\left( \sum_{k\in I(n)}(k+1)^2|\widehat{g}(k+1)|^2 \right)^ {\frac{\tp}{2}}  \right)^{1-\frac{2}{p}}
\\ & \asymp  \|g'\|^{2}_{A^{2,\tilde{p}}_{\frac{\tilde{p}}{2}}}  \| \hti(f)\|^2_{H^p}
\\ & \lesssim \|g'\|^{2}_{A^{2,\tilde{p}}_{\frac{\tilde{p}}{2}}} \| f\|^2_{H^p}.
\end{split}\end{equation*}
 So, $\hg: H^p\to H^2$ is bounded and $\| \hg\|_{H^p\to H^2}\lesssim \|g'\|_{A^{2,\tilde{p}}_{\frac{\tilde{p}}{2}}}$. Furthermore, since the polynomials are dense in
$A^{2,\tilde{p}}_{\frac{\tilde{p}}{2}}$, then for any $g\in A^{2,\tilde{p}}_{\frac{\tilde{p}}{2}}$, there exists polynomials $\{P_n\}$ such that $\lim_{n\to\infty}\|P_n'-g'\|_{A^{2,\tilde{p}}_{\frac{\tilde{p}}{2}}}=0$. Hence, we have
$$
\lim_{n\to\infty}\|\H_g-\H_{P_n}\|_{H^p\to H^2}\lesssim \lim_{n\to\infty}\|P_n'-g'\|_{A^{2,\tilde{p}}_{\frac{\tilde{p}}{2}}}=0.
$$
Since $\H_{P_n}$ is finite rank, $\H_g$ is compact.

Reciprocally,  assume that $\hg: H^p\to H^2$ is bounded. Let be $0<\rho<1$. We will make use of Lemma~\ref{le:technical} in  the case $q=2$, $s=\frac{\tp}{p}$ and $t=\frac{\tp}{2p}$. That is,
$$\phi_\rho(r)=\left( M_2(r,g'_\rho)(1-r)^{1/2}\right)^{\frac{\tp}{p}}, \quad 0\le r<1.$$
Next,
\begin{equation}\begin{split}\label{eq:dilated}
 \| \mathcal{H}_{g_{\rho}}(Q_\rho)\|^2_{H^2} & = \sum_{n=0}^\infty \sum_{k\in I(n)}  (k+1)^{2}|\widehat{g}(k+1)|^2 \rho^{2k+2} \left(\int_{0}^1 t^k Q_\rho(t)\,dt  \right)^2
\\ & \ge  \sum_{n=0}^\infty \left(\int_{0}^1 t^{2^{n+1}} Q_\rho(t)\,dt  \right)^2 \sum_{k\in I(n)}  (k+1)^{2}|\widehat{g}(k+1)|^2  \rho^{2k+2}
\\ & =  \sum_{n=0}^\infty \left(\int_{0}^1 t^{2^{n+1}} Q_\rho(t)\,dt  \right)^2  \| \Delta_n (g_\rho)'\|^2_{H^2}.
\end{split}\end{equation}
 Moreover, by Lemma~\ref{le:technical}, the M. Riesz projection theorem and \cite[Lemma~10]{PelRathg}
\begin{equation}\begin{split}\label{eq:dilated2}
\int_{0}^1 t^{2^{n+1}} Q_\rho(t)\,dt & \ge
\int_{1-2^{-(n+1)}}^1 t^{2^{n+1}} Q_\rho(t)\,dt
\\ & \gtrsim \int_{1-2^{-(n+1)}}^1  Q_\rho(t)\,dt
\\ & \gtrsim  M^{\frac{\tp}{p}}_2 \left( 1-2^{-(n+1)} ,(g_\rho)' \right )\int_{1-2^{-(n+1)}}^1  (1-t)^{\frac{\tp}{2p}}\,dt
\\ & \asymp    M^{\frac{\tp}{p}}_2 \left( 1-2^{-(n+1)} , (g_\rho)' \right) 2^{-n\left( \frac{\tp}{2p}+1\right) }
\\ & \gtrsim M^{\frac{\tp}{p}}_2 \left( 1-2^{-(n+1) }, \Delta_n (g_\rho)' \right) 2^{-n\left( \frac{\tp}{2p}+1\right) }
\\ & \gtrsim  \| \Delta_n (g_\rho)'\|^{\frac{\tp}{p}}_{H^2}2^{-n\left( \frac{\tp}{2p}+1\right) }.
\end{split}\end{equation}
Therefore, joining \eqref{eq:dilated}, \eqref{eq:dilated2} and \eqref{eq:symbolnorm} it follows that
\begin{equation*}\begin{split}
\| \mathcal{H}_{g_{\rho}}(Q_\rho)\|^2_{H^2} & \gtrsim   \sum_{n=0}^\infty 2^{-n\left( \frac{\tp}{p}+2\right) }  \| \Delta_n (g_\rho)'\|^{2\left(1+\frac{\tp}{p}  \right)}_{H^2}
\\ & =  \sum_{n=0}^\infty 2^{-n\left( \frac{\tp}{2}+1\right) }  \| \Delta_n (g_\rho)'\|^{\tp   }_{H^2}
\\ & \asymp \|(g_\rho)'\|^{\tp}_{A^{2,\tilde{p}}_{\frac{\tilde{p}}{2}}}.
\end{split}\end{equation*}
Next, observe that \eqref{eq:dilated} implies that  $\| \mathcal{H}_{g}(Q_\rho)\|_{H^2}\ge  \| \mathcal{H}_{g_{\rho}}(Q_\rho)\|_{H^2}$ for any $0< \rho <1$.
Consequently, using the above inequalities and Lemma~\ref{le:technical}
\begin{equation*}\begin{split}
\| \hg\|^2_{H^p\to H^2} & \ge \frac{ \| \mathcal{H}_{g}(Q_\rho)\|^2_{H^2} }{\| Q_\rho\|^2_{H^p}}
 \ge  \frac{\| \mathcal{H}_{g_\rho}(Q_\rho)\|^2_{H^2} }{\| Q_\rho\|^2_{H^p}}
 \gtrsim   \frac{ \|(g_\rho)'\|^{\tp}_{A^{2,\tilde{p}}_{\frac{\tilde{p}}{2}}}}{\| Q_\rho\|^2_{H^p}}
\\ & \gtrsim   \frac{ \|(g_\rho)'\|^{\tp}_{A^{2,\tilde{p}}_{\frac{\tilde{p}}{2}}}}{\| \phi_\rho\|^2_{L^p([0,1))}}
 =    \frac{ \|(g_\rho)'\|^{\tp}_{A^{2,\tilde{p}}_{\frac{\tilde{p}}{2}}}}{ \|(g_\rho)'\|^{\frac{2\tp}{p}}_{A^{2,\tilde{p}}_{\frac{\tilde{p}}{2}}}}= \|(g_\rho)'\|^{2}_{A^{2,\tilde{p}}_{\frac{\tilde{p}}{2}}},
\end{split}\end{equation*}
that is,
$$  \|(g_\rho)'\|_{A^{2,\tilde{p}}_{\frac{\tilde{p}}{2}}}\lesssim \| \hg\|_{H^p\to H^2}, \quad 0<\rho<1.$$
So, letting $\rho\to 1^{-}$, it follows that
$g'\in A^{2,\tilde{p}}_{\frac{\tilde{p}}{2}}$ and $
\| \hg\|_{H^p\to H^2}\gtrsim \|g'\|_{A^{2,\tilde{p}}_{\frac{\tilde{p}}{2}}}.
$
 This finishes the proof.
\end{Prf}

\section{Proof of Theorem~\ref{thm:H2target} and some consequences.}

\begin{Prf}{\em{ Theorem~\ref{thm:H2target}. }}
Part (i) follows from Theorem~\ref{OnH1}.
Part (ii) follows from Theorem~\ref{th:hardypmenor2}.
 Part (iii) follows from Theorem~\ref{th:hardypmayor2}. It remains to prove Part (iv). Observe that $\hg(1)=\frac{g(z)-g(0)}{z}$. So if $\hg: H^\infty \to H^2$ is bounded, then $g\in H^2$ and $\| g-g(0)\|_{H^2}\le  \| \hg\|_{H^\infty\to H^2}.$ Reciprocally, if $g\in H^2$, then
\begin{equation*}\begin{split}
\| \hg(f)\|^2_{H^2}= & \sum_{k=0}^\infty  (k+1)^{2}|\widehat{g}(k+1)|^2 \left|\int_{0}^1 t^k f(t)\,dt  \right|^2
\\ \le & \|f\|^2_{H^\infty} \sum_{k=0}^\infty   (k+1)^{2}|\widehat{g}(k+1)|^2 \left(\int_{0}^1 t^k   dt\right)^2=  \|f\|^2_{H^\infty}\| g-g(0)\|_{H^2}.
\end{split}\end{equation*}
Therefore, $\hg: H^\infty \to H^2$ is bounded and $\| \hg\|_{H^\infty\to H^2}\le \| g-g(0)\|_{H^2} $. This, together with the fact that the analytic polynomials are dense in $H^2$ and any generalized Hilbert operator induced by a polynomial is finite rank, shows $\H_g$ is also compact. This finishes the proof.
\end{Prf}
\vskip2mm
Now we will prove {Corollary~\ref{th:counterexample}. With this aim
we recall that for  $\lambda>1$,  $g\in \H(\D)$ is called a lacunary series of step $\lambda$ if $g(z)=\sum_{k=0}^\infty a_k z^{n_k}$ and  $n_{k+1}\geq \lambda n_k$ for all $k$.

Prior to prove Theorem~\ref{th:counterexample} let us observe that for
$2<p<\infty$ there is a lacunary series $g\in \H(\D)$ such that  $g\in \Lambda (p,1/p)$ and
 $g'\notin A^{2,\tilde{p}}_{\frac{\tilde{p}}{2}}$, where  and  $\frac{1}{\tp}=\frac{1}{2}-\frac{1}{p}$. In fact,
take $$g(z)=\sum_{n=0}^\infty 2^{-\frac{n}{p}} z^{2^n}.$$ By Theorem~\ref{th:mlip}, $g\in \Lambda (p,1/p)$.
Moreover,
$ \sum_{n=0}^\infty 2^{-n \left( \frac{\tp}{2}+1 \right)} 2^{n\tp \left(1-\frac{1}{p}\right)}= \sum_{n=0}^\infty 1=\infty.$
Therefore,   $g'\notin A^{2,\tilde{p}}_{\frac{\tilde{p}}{2}}$ by \eqref{eq:symbolnorm}.

\vskip2mm
\begin{Prf}{\em{Corollary~\ref{th:counterexample}.}}
Take  $g\in \H(\D)$ a lacunary series (of step $\lambda>1$)  such that $g\in \Lambda (p,1/p)$ and
 $g'\notin A^{2,\tilde{p}}_{\frac{\tilde{p}}{2}}$.  Then, by  Theorem~\ref{th:hardypmayor2}, $\hg$ is not bounded from $H^p$ to $H^2$. Going further,
by \eqref{Hgcoef} for any $f\in H^1$, the analytic function $\hg(f)$ is also a lacunary series with the same step $\lambda$. So, by \cite[Vol I, v, Theorem 8.20]{Zy} there is
$A_{\lambda}>0$   such that
$$ A_\lambda \|\h_g(f)\|_{H^2}\le \|\h_g(f)\|_{H^1}, \quad f\in H^1.$$
This inequality together with the fact that $\hg$ is not bounded from $H^p$ to $H^2$, implies that $\hg$ is not bounded from $H^p$ to $H^1$.
This finishes the proof.
\end{Prf}

\vskip2mm

Finally we will prove Theorem~\ref{th:lacunaryp-p}.

\begin{Prf}{\em{Theorem~\ref{th:lacunaryp-p}.}}
Assume that $\hg$ is bounded on $H^p$. Then it is bounded from $H^p$ to $H^2$, so  by  Theorem~\ref{th:hardypmayor2},  $g'\in A^{2,\tilde{p}}_{\frac{\tilde{p}}{2}}$ and
$$ \|\h_g\|_{H^p\to H^p}\ge \|\h_g\|_{H^p\to H^2}\gtrsim \| g'\|_{ A^{2,\tilde{p}}_{\frac{\tilde{p}}{2}}}.$$
On the other hand, assume that   $g$ is a lacunary series (with $\lambda>1$) such
that  $g'\in A^{2,\tilde{p}}_{\frac{\tilde{p}}{2}}$.
Then for any $f\in H^p$, the analytic function $\hg(f)$ is also a lacunary series with the same step $\lambda$.
So, by \cite[Vol I, v, Theorem 8.20]{Zy} and Theorem~\ref{th:hardypmayor2} there is
$A_{\lambda,p}>0$   such that
$$ \|\h_g(f)\|_{H^p}\le  A_{\lambda,p} \|\h_g(f)\|_{H^2}\lesssim A_{\lambda,p} \| g'\|_{ A^{2,\tilde{p}}_{\frac{\tilde{p}}{2}}} \|f\|_{H^p},  \, \quad f\in H^p.$$
That is, $\hg$ is bounded on $H^p$ and
$ \|\h_g\|_{H^p\to H^p}\lesssim  \ \| g'\|_{ A^{2,\tilde{p}}_{\frac{\tilde{p}}{2}}}.$
This finishes the proof.
\end{Prf}

It is worth mentioning that the sufficient condtion $g'\in A^{2,\tilde{p}}_{\frac{\tilde{p}}{2}}$ for the boundedness of $\hg$ on $H^p$, $p>2$, coincides with the conditon
$g'\in A^{p,\tilde{p}}_{\frac{\tilde{p}}{2}}$ if $g$ is a lacunary series. Finally, let us observe that \cite[Theorem~4.7] {GuoTang2026} states that for any $g\in\H(\D)$ the condition $g'\in A^{p,\tilde{p}}_{\frac{\tilde{p}}{2}}$ is sufficient,  but not necessary, for $\hg$ to be bounded on $H^p$, $p>2$.



\begin{thebibliography}{99}
 %
\bibitem{BSS} P.~S.~ Bourdon, J.~H.~Shapiro and W.~T.~Sledd,
 Fourier series, mean Lipschitz spaces and bounded mean oscillation,
 Analysis at Urbana, 1 (1986), LMS Lecture Note Series n. 137, 81--110, (1989).
 %

\bibitem{Blasco2022} O.~Blasco, Remarks on generalized Hilbert operators,
J. Math. Sci. (N.Y.) 266 (2022), no. 2, 274–284.



%
%

%
\bibitem{D} P.~L. Duren, Theory of $H\sp{p} $ Spaces, Academic Press,
New~York-London 1970. Reprint: Dover, Mineola, New York 2000.
%
%
\bibitem{Flett} T.~M.~Flett, {\em{The dual of an inequality of Hardy and Littlewood and some
related inequalities}}, J. Math. Anal. Appl.  \textbf{38} (1972),
756--765.


\bibitem{GaGiPeSis}     P. Galanopoulos,  D. Girela, J. A. Pel\'aez and A. Siskakis, Generalized Hilbert operators,
		Ann. Acad. Sci. Fenn. Math.  39 (2014), 231--258.

%

\bibitem{GuoTang2026} Y.~Guo and P.~Tang, Generalized Hilbert operators on Hardy spaces,
https://arxiv.org/abs/2607.28221

%
 %

 \bibitem{HL-32} G. H. Hardy and J. E. Littlewood,
{\em{Some properties of fractional integrals, II}}, Math. Z.
\textbf{34} (1932), 403--439.


%
%

%

%
\bibitem{MP} M.~Mateljevic and M.~Pavlovic, {\em{$L^p$ behaviour of the integral means
of analytic functions}}, Studia Math. \textbf{77} (1984), 219--237.
%

%
 \bibitem{Pabook} M.~Pavlovi\'c, \emph{Introduction to function spaces on the
disk, Posebna Izdanja [Special Editions]\/}, vol. $20$, Matemati\v
cki Institut SANU, Beograd, 2004.
%
 \bibitem{Padec} M.~Pavlovi\'c, {\em{Analytic functions with decreasing coefficients and Hardy spaces}}, preprint, (2011).
%

		\bibitem{PelRathg} J.~\'A. Pel\'aez and J. R\"atty\"a, Generalized Hilbert operators on weighted Bergman spaces, Adv. Math.  240 (2013), 227--267.






\bibitem{PelSeco} J.~\'A. Pel\'aez and D.~Seco, Schatten classes of generalized Hilbert operators,
Collect. Math. 69 (2018), no. 1, 83–105.






\bibitem{Zhu} K.~Zhu, \emph{Operator Theory in Function Spaces\/},
2nd edition. Math Surveys and Monographs, Vol. 138,
  American Mathematical Society, Providence,
RI, 2007.
%
\bibitem{Zy} A.~Zygmund, \emph{Trigonometric Series\/},
Vol. I and Vol. II, Second edition, Camb. Univ. Press, Cambridge,
1959.
\end{thebibliography}
\end{document}